\documentclass[12pt]{amsart}
\usepackage[colorlinks=true, pdfstartview=FitV, linkcolor=blue,
citecolor=blue,
urlcolor=blue,
breaklinks=true]{hyperref}
\usepackage{amsmath,amsfonts,amsthm, amssymb, mathrsfs, etoolbox, enumerate, tikz-cd}
\usepackage[all]{xy}
\usepackage[perpage,bottom]{footmisc}

\newcommand\C{\mathbb{C}}
\newcommand\Q{\mathbb{Q}}
\newcommand\Z{\mathbb{Z}}

\newcommand\PP{\mathbb{P}}

\newcommand\osp{\frak{osp}}

\newcommand\gl{\frak{gl}}
\newcommand\g{\frak{g}}
\newcommand\B{\frak{b}}
\newcommand\h{\frak{h}}

\DeclareMathOperator{\Hom}{Hom}
\DeclareMathOperator{\End}{End}

\newtheorem{theorem}{Theorem}[section]
\newtheorem{lemma}[theorem]{Lemma}
\newtheorem{defin}[theorem]{Definition}

\newtheorem{corollary}[theorem]{Corollary}
\newtheorem{prop}[theorem]{Proposition}
\newtheorem{prob}[theorem]{Problem}
\newtheorem{rmrk}[theorem]{Remark}

\numberwithin{equation}{section}

\allowdisplaybreaks

\newtoggle{details}

\iftoggle{details}{%
  \newcommand{\details}[1]{
      \ \\
      {\color{blue}
        \textbf{Details:} #1
      }
  }
}{%
  \newcommand{\details}[1]{}
}

\title[Explicit formulas for Eigenvalues of Capelli operators]{Explicit formulas for Eigenvalues of Capelli operators for the Lie superalgebra $\osp(1|2n)$}
\author{Dene Lepine\textsuperscript{}}
\address{Department of Mathematics and Statistics, University of Ottawa}
\email{dlepi025@uottawa.ca}

\begin{document}

\begin{abstract}
	We define a natural basis for the algebra of $\frak{gosp}(1|2n)$-invariant differential operators on the affine superspace $\C^{1|2n}$. We prove that these operators lie in the image of the centre of the enveloping algebra of $\frak{gosp}(1|2n)$. Using this result, we compute explicit formulas for the eigenvalues of these operators on irreducible summands of $\mathcal{P}(\C^{1|2n})$. This settles the Capelli eigenvalue problem for orthosymplectic Lie superalgebras in the cases that were not addressed in \cite{SSS20} and \cite{SSS21}. Our main technique relies on an explicit calculation of a certain determinant with polynomial entries.
\end{abstract}

\maketitle

\tableofcontents
\thispagestyle{empty}

\addtocontents{toc}{\protect\setcounter{tocdepth}{1}}

%
%
\section{Introduction}
%
%

\footnotetext[1]{This paper grew out of the results of my Masters Thesis at the University of Ottawa under the supervision of Professor Hadi Salmasian. The author would like thank Professor Siddhartha Sahi for his suggestions on an earlier draft of this article, which improved the presentation of Theorem \ref{CapEvalThm}. He also thanks J\'{e}r\'{e}my Champagne for the useful discussions that eventually lead to the idea for Lemma \ref{Njlemma}.}

The Capelli identity, discovered by Alfredo Capelli \cite{Cap87}, is an equality between a differential operator on the space of $n\times n$ matrices, and a non-commutative determinant whose entries are natural generators of the universal enveloping algebra of $\gl(n)$. This identity played an integral role in Hermann Weyl's  proof of the FFT in invariant theory \cite{Weyl39}. Howe and Umeda's seminal work \cite{HowUme91} transcends this identity from the viewpoint of multiplicity-free actions and formulates the \emph{abstract Capelli problem}, see Problem \ref{AbsCapProb}. In \cite{Sah94}, Sahi introduced a natural basis (the \emph{Capelli operators}) for the algebra of invariant polynomial coefficient differential operators on a multiplicity-free module of a reductive Lie algebra, and a multi-parameter family of polynomials whose special values yield the eigenvalues of these operators. The Capelli operators of \cite{Sah94} include the Capelli element as a special case. The \emph{Capelli eigenvalue problem} is the problem of describing the eigenvalues of this natural basis, see Problem \ref{CapEvalProb}. Other examples of Capelli operators were explored by Konstant, Sahi, and Wallach, see \cite{KosSah91} and \cite{Wal92}.

The solution to the Capelli eigenvalue problem was followed by a further investigation of \emph{interpolation Jack polynomials} by Sahi, Knop-Sahi, and around the same time by Okounkov-Olshanski \cite{Sah94},~\cite{KnoSah96},\cite{OkOl97}. In the past few years, the Capelli eigenvalue problem has emerged again in the setting of Lie superalgebras, and nearly completely solved for multiplicity-free actions that arise from the Tits-Kantor-Koecher construction for Jordan superalgebras~\cite{ASS18}, \cite{SalSah16}, \cite{SSS20}, \cite{SSS21}. However, for a small number of cases the uniform methods used in the aforementioned works did not apply. In particular, for the Lie superalgebra $\frak{gosp}(m,2n):=\C\oplus \osp(m,2n)$ acting on its natural representation $\C^{m|2n}$, which corresponds to the Jordan superalgebra $(m-1,2n)_+$ in Kac's notation \cite{Kac77}, the problem was solved for all cases in \cite{SSS20} and \cite{SSS21}, except for when $m=1$ and $n>0$. The reason why the case $m=1$ is different is that the highest weights of the irreducible components in the decomposition of $\mathcal{S}(\C^{m|2n})$ into $\frak{gosp}(m|2n)$-modules look different from the generic case: they are of the form $\delta_1+\dotsb+\delta_i$, unlike the generic case which has highest weights of the form $k\delta_1$. Furthermore when $m=1$, the number of irreducible components in $\mathcal{S}^d(\C^{1|2n})$ does not tend to infinity. Therefore, a priori it is not clear how to parametrize irreducible components suitably by partitions to obtain a polynomial formula for the eigenvalues of the Capelli operators. Another issue is that the general argument of \cite{SSS20} for surjectivity of the map from the centre of the enveloping algebra $\frak{U}(\frak{gosp}(1|2n))$ onto the algebra of invariant differential operators fails. This surjectivity is a crucial step in the proofs of \cite{SSS20}.

In this paper, we circumvent the above issues and give solutions to both the Capelli eigenvalue problem and the abstract Capelli problem for the Lie superalgebra $\osp(1|2n)$, see Theorem \ref{CapEvalThm} and Theorem \ref{MainIntroTheorem} respectively. The two new ideas that arise are a proper indexing of the irreducible components by partitions and an explicit formula for the determinant of a matrix whose entries are the eigenvalues of the operators $C^i E^j$, where $C$ is the Casimir operator and $E$ is the Euler operator.

%
%
\section{Main Results}
%
%

Let $V:=\C^{1|2n}$ be the super vector space over $\C$ of dimension $1|2n$. Equip $V$ with the non-degenerate even supersymmetric bilinear form $B\colon V\times V\to \C$ given by the $(1|n+n)$-block matrix
	\begin{align*}
		\frak{J}_{1|2n}&=\begin{bmatrix}
			1&0&0\\
			0&0&I_n\\
			0&-I_n&0
		\end{bmatrix}.
	\end{align*}
Denote by $\osp(1|2n)$ the orthosymplectic Lie superalgebra that leaves $B$ invariant. Let $\h$ be the standard Cartan subalgebra of $\osp(1|2n)$ with basis $\mathscr{B}=\lbrace E_{i+1,i+1}-E_{i+n+1,i+n+1}\rbrace_{i=1}^n$, where the $E_{i,j}$ denote the usual matrix units. Let $\lbrace \delta_1,\dotsc, \delta_n\rbrace$ be the basis of $\h^*$ dual to $\mathscr{B}$. Let $\B$ be the Borel subalgebra of $\osp(1|2n)$ corresponding to the fundamental system
	\begin{align*}
		\Sigma:=\lbrace \delta_i-\delta_{i+1}\rbrace_{i=1}^{n-1}\cup\lbrace \delta_n \rbrace.
	\end{align*}
Set
	\begin{align*}
		\g:=\frak{gosp}(1|2n):=\C Z\oplus \osp(1|2n),
	\end{align*}
where $Z$ is a central element of $\g$. We can extend the natural $\osp(1|2n)$-module structure on $V$ to $\g$ by defining the action of $Z$ to be $-I_{V}$. 

Let $\mathcal{P}(V)$ be the superalgebra of (super)polynomials on $V$. We denote the generators of $V$ dual to the standard basis of $V$ by $\lbrace y,x_1,\dotsc,x_{2n}\rbrace$, where $y$ is even and $x_1,\dotsc,x_{2n}$ are odd. Define the (super)derivations $\lbrace \partial_y,\partial_1,\dotsc, \partial_{2n} \rbrace$ for all $1\leq i,j\leq 2n$ and homogeneous $u,v\in\mathcal{P}(V)$ by the relations
	\begin{align*}
		\partial_y(y)&:=1,&&\partial_y(x_j):=0,&\partial_y(uv)&:=\partial_y(u)v+u\partial_y(v),\\
		\partial_i(y)&:=0,&&\partial_i(x_j):=\delta_{i,j},&\partial_i(uv)&:=\partial_i(u)v+(-1)^{|u|}u\partial_i(v).
	\end{align*}
Let $\mathcal{D}(V)$ be the superalgebra of constant-coefficient (super)differential operators generated by $\lbrace \partial_y,\partial_1,\dotsc, \partial_{2n} \rbrace$ and let $\mathcal{PD}(V)$ be the superalgebra of (super)polynomial coefficient (super)differential operators. We can consider $\mathcal{PD}(V)$ as the subalgebra of $\End(\mathcal{P}(V))$ generated by $y,x_1,\dotsc,x_{2n},\partial_y,\partial_1,\dotsc,\partial_{2n}$. As a vector space we have $\mathcal{PD}(V)\cong \mathcal{P}(V)\otimes \mathcal{D}(V)$ and we use $\mathcal{PD}^k(V)$ to denote the image of $\mathcal{P}^k(V)\otimes \mathcal{D}^k(V)$ in $\mathcal{PD}(V)$, where $\mathcal{P}^k(V)$ and $\mathcal{D}^k(V)$ are the homogenous degree $k$ components. For more details on our notation see \cite[Sec. 2]{SalSah16}.

The $\g$-module structure on $V$ induces a $\g$-module structure on $\mathcal{P}(V)$. Explicitly, the action of $\osp(1|2n)$ is given by the restricted action from $\gl(1|2n)$ to $\osp(1|2n)$, where the action of $\gl(1|2n)$ on $\mathcal{P}(V)$ is defined for all $1\leq i,j\leq 2n$ by
	\begin{align*}
		E_{1,1}&\mapsto y\partial_y,&E_{1,j+1}&\mapsto y\partial_j,\\
		E_{j+1,1}&\mapsto x_j\partial_y,&E_{i+1,j+1}&\mapsto x_i\partial_j.
	\end{align*}
Furthermore, $Z$ acts on $\mathcal{P}(V)$ by the \emph{degree operator}. That is,
	\begin{align*}
		Z\mapsto E:=y\partial_y+\sum_{i=1}^{2n}x_i\partial_i.
	\end{align*}
Since $V$ is self-dual, then $\mathcal{D}(V)\cong \mathcal{P}(V^*)\cong \mathcal{P}(V)$. Thus, the action of $\g$ on $\mathcal{P}(V)$ induces an action on $\mathcal{D}(V)$ and hence on $\mathcal{PD}(V)\cong \mathcal{P}(V)\otimes \mathcal{D}(V)$. Observe that the action of $\g$ on $\mathcal{P}(V)$ induces a map
	\begin{align*}
		\frak{U}(\g)\to \mathcal{PD}(V),
	\end{align*}
where $\frak{U}(\g)$ denotes the universal enveloping algebra of $\g$. An element $D\in \mathcal{PD}(V)$ is $\g$-invariant if and only if $Dx=(-1)^{|D||x|}xD$, for all $x\in \g$. Denote by $\mathcal{PD}(V)^\g$ the subalgebra of $\g$-invariant polynomial coefficient differential operators. Notice that the image of the centre of $\frak{U}(\g)$, denoted by $\mathcal{Z}(\g)$, lies in $\mathcal{PD}(V)^\g$. That is, we have a map
	\begin{align}\label{AbsCapProb}
		\mathcal{Z}(\g)\to\mathcal{PD}(V)^\g.
	\end{align}

\begin{prob}[Abstract Capelli Problem for $\osp(1|2n)$]\label{AbstractCapelliProblem}
	Determine whether or not the map \eqref{AbsCapProb} is surjective. 
\end{prob}

To put Problem \ref{AbstractCapelliProblem} in perspective, recall that in the non-super setting, this problem was investigated in the work of Howe and Umeda~\cite{HowUme91}. In \cite{SSS20} this problem was studied in the general setting of multiplicity-free modules of basic classical Lie superalgebras that arise from the super Tits-Kantor-Koecher construction. More precisely, in \cite[Corollary 1.18]{SSS20} the authors addressed Problem \ref{AbstractCapelliProblem} for pairs $(\g,V)$ associated to a simple unital Jordan superalgebra $J$ such that $J_1\neq \lbrace 0\rbrace$ and $J$ is not of type $(0,2n)_+$ or $JP(0,n)$, in the Kac notation. The answer turns out to be affirmative for all such pairs $(\g,V)$, except when $V$ is a $\frak{gosp}(2|4)$-module constructed from the exceptional Jordan superalgebra of type $F$. The proof technique of \cite{SSS20} uses the connection between the image of the Harish-Chandra homomorphism and Sergeev-Veselov interpolation polynomials. This technique is not available when $\g=\osp(1|2n)$ and $V=\C^{1|2n}$, corresponding to $J=(0,2n)_+$.

Our first main result of this paper is an affirmative answer to Problem \ref{AbstractCapelliProblem} (see Theorem \ref{MainIntroTheorem}). This resolves the abstract Capelli problem for pairs $(\g,V)$ associated to the Jordarn superalgebra $(0,2n)_+$ which were left out in \cite{SSS20}.\\

To explain our second main result, we begin by constructing the Capelli basis of $\mathcal{PD}(V)^\g$. Let us define two operators $R^2$ and $\nabla^2$ on $\mathcal{P}(V)$ as follows: $R^2$ is the operator of left multiplication by $y^2-2\sum_{1\leq i\leq n} x_{i+n}x_i$ and $\nabla^2$ is the constant coefficient differential operator $\partial_y-2\sum_{1\leq i\leq n} \partial_{i+n}\partial_i.$ We call $\nabla^2$ the \emph{super Laplace operator}. These operators admit the following relations
	\begin{align*}
		\left[ (1/2)R^2, -(1/2)\nabla^2 \right]&= E+\frac{1-2n}{2},\\
		\left[ E+\frac{1-2n}{2}, (1/2)R^2 \right]&= R^2,\\
		\left[ E+\frac{1-2n}{2},-(1/2)\nabla^2 \right]&=\nabla^2.
	\end{align*}
That is, the elements $\left\lbrace E+\frac{1-2n}{2},(1/2)R^2, -(1/2)\nabla^2\right\rbrace$ form a standard $\frak{sl}_2(\C)$-triple. Furthermore, one can check directly that the actions of $\frak{sl}_2(\C)$ and $\osp(1|2n)$ commute, so that $\mathcal{P}(V)$ is an $\osp(1|2n)\times \frak{sl}_2(\C)$-module. There is a natural grading of $\mathcal{P}(V)$ by degree and we have $Ep=kp$ for $p\in \mathcal{P}^k(V)$. Furthermore, define $\mathcal{H}=\ker (\nabla^2)$ and $\mathcal{H}_k=\mathcal{P}^k(V)\cap \mathcal{H}$. The elements of $\mathcal{H}$ are called \emph{harmonic superpolynomials}. The next Lemma is a special case of \cite[Th. 5.2]{Cou13}.

\begin{lemma}\label{HkIrred}
	For $0\leq k\leq 2n+1$, the space $\mathcal{H}_k$ is an irreducible $\osp(1|2n)$-module of highest weight
	\begin{align*}
		\lambda_k&=\sum_{j=1}^{\min(2n-k+1,k)}\delta_j.
	\end{align*}
\end{lemma}

Using Lemma \ref{HkIrred} and the fact that multiplication by $R^2$ commutes with the $\osp(1|2n)$ action it follows that $R^{2\ell}\mathcal{H}_k$ is an irreducible $\osp(1|2n)$-module of highest weight $\lambda_k$, for all $\ell\in \Z_{\geq 0}$. For the following Lemma see \cite[Lemma 2.1]{Cou13}.

\begin{lemma}[Fischer Decomposition]\label{FischerDecomp}
	For $k\in \Z_{\geq 0}$ and $\ell=\lfloor\frac{k}{2}\rfloor$, the decomposition of $\mathcal{P}(V)$ as an $\osp(1|2n)$-module is given by
		\begin{align*}
			\mathcal{P}^k(V)&\simeq \begin{cases}
				\mathcal{H}_k\oplus R^2\mathcal{H}_{k-2}\oplus \dotsb\oplus R^{2\ell}\mathcal{H}_{k-2\ell}& \textrm{ for }k\leq 2n+1\textrm{ and }\\
				R^2\mathcal{P}^{k-2}(V)&\textrm{ for }k>2n+1.
			\end{cases}
		\end{align*}
\end{lemma}

Thus as a $\g$-module, the space $\mathcal{P}(V)$ admits a multiplicity-free decomposition. Define $$\mathbb{P}:=\lbrace (\nu_1,\nu_2)\colon \nu_1,\nu_2\in \Z_{\geq 0},~\nu_1\geq \nu_2\rbrace,$$ and for any $\nu\in \mathbb{P}$ set $|\nu|:=\nu_1+\nu_2$. We will parametrize the irreducible submodules of $\mathcal{P}^k(V)$ by the set
	\begin{align}\label{IrredInK}
		\mathbb{P}_{k,n}^*&:=\left\lbrace \nu\in \mathbb{P}\colon |\nu|=k\textrm{ and }\nu_2\geq \left\lfloor \frac{k}{2}\right\rfloor-n\right\rbrace,
	\end{align}
via the assignment $\nu\mapsto V_\nu:=R^{2\nu_2}\mathcal{H}_{\nu_1-\nu_2},$ for $\nu\in \mathbb{P}_{k,n}^*$. Then we have
	\begin{align}\label{Pdecomp}
		\mathcal{P}^k(V)=\bigoplus_{\nu\in \PP_{k,n}^*} V_\nu.
	\end{align}
Let $\mathcal{D}^k(V)\otimes \mathcal{P}^k(V)\to \C$ be the canonical non-degenerate pairing between $\mathcal{D}^k(V)$ and $\mathcal{P}^k(V)$. This pairing induces a $\g$-module isomorphism $\mathcal{D}^k(V)\simeq \mathcal{P}^k(V)^*.$ That is,
	\begin{align}\label{Ddecomp}
		\mathcal{D}^k(V)\simeq \bigoplus_{\nu\in \PP_{k,n}^*}V_\nu^*.
	\end{align}
We remark that the latter isomorphism can be chosen such that for every $\nu\in \mathbb{P}_{k,n}^*$, the restriction of the natural map $\mathcal{D}^k(V)\otimes \mathcal{P}^k(V)\to \C$ to $V_\nu^*\otimes V_\nu$ is the standard duality pairing. Define
	\begin{align}\label{IrredUptoD}
	\Lambda_{d,n}^*:=\bigcup_{k=0}^d \PP_{k,n}^*\quad\textrm{and}\quad \Lambda_n^*:=\bigcup_{k\geq 0} \PP_{k,n}^*.
	\end{align}
Using \eqref{Pdecomp} and \eqref{Ddecomp} we have 
	\begin{align}\label{CapIso}
		\mathcal{PD}(V)^\g\simeq\bigoplus_{\nu,\nu'\in \Lambda_n^*}( V_\nu\otimes V_{\nu'}^* )^\g\simeq\bigoplus_{\nu,\nu'\in \Lambda_{n}^*}\Hom_\g(V_{\nu'},V_\nu)\simeq \bigoplus_{\nu\in \Lambda_{n}^*}\Hom_\g(V_{\nu},V_\nu).
	\end{align}

\begin{defin}[Capelli Operator]
	For $\nu\in \Lambda_n^*$, define $D_\nu\in \mathcal{PD}(V)^\g$ to be the element mapped to $I_\nu\in \Hom_\g(V_\nu,V_\nu)$ by \eqref{CapIso}. The operator $D_\nu$ is called a \emph{Capelli Operator}.
\end{defin}
	
Indeed, Schur's Lemma implies $\dim\Hom_\g(V_\nu,V_\nu)=1$ and therefore the family $\lbrace D_\nu\rbrace_{\nu\in \Lambda_n^*}$ forms a basis of $\mathcal{PD}(V)^\g$. Schur's Lemma also implies that for any $\mu,\nu\in \Lambda_{n}^*$ the Capelli operator $D_\mu$ acts on $V_\nu$ by a scalar. This gives rise to the following problem called the \emph{Capelli Eigenvalue problem} for $\osp(1|2n)$:

\begin{prob}[Capelli Eigenvalue Problem for $\osp(1|2n)$]\label{CapEvalProb}
	Let $\mu,\nu\in \Lambda_n^*$. Determine the scalar $c_{\nu}(\mu)$ by which $D_\nu$ acts on $V_\mu$.
\end{prob}

Our answer to Problem \ref{CapEvalProb} makes use of a special case of \emph{Knop-Sahi polynomials}. These polynomials are uniquely defined by certain vanishing conditions. In particular, we let $\rho:=(-n-\frac{1}{2},0)$ and let $P_\nu^\rho$, for $\nu\in \mathbb{P}$, be the polynomial defined in \cite{KnoSah96} that satisfies the following conditions:
		\begin{enumerate}[(1)]
			\item $P_\nu^\rho$ is symmetric in two variables $x$ and $y$,
			\item $\deg(P^\rho_\nu)\leq |\nu|$,
			\item $P_\nu^\rho(\mu+\rho)=0$ for all $\mu\in \mathbb{P}$ with $|\mu|\leq |\nu|$ and $\mu\neq \nu$,
			\item $P_\nu^\rho(\nu+\rho)=(\nu_1-\nu_2)!(\nu_2!)(\nu_1-(n+\frac{1}{2}))^{\underline{\nu_2}}$,
		\end{enumerate}
	where $x^{\underline{i}}=x(x-1)\dotsb(x-i+1)$.

\begin{theorem}\label{CapEvalThm}
	Let $\mu,\nu\in \Lambda^*_n$. Then $c_\nu(\mu)=\frac{P_\nu^\rho(\mu_1-n-\frac{1}{2},\mu_2)}{(\nu_1-\nu_2)!(\nu_2!)(\nu_1-(n+\frac{1}{2}))^{\underline{\nu_2}}}$. In particular, the explicit formula is given by $c_\nu (\mu)=ab$, where
		\begin{align*}
			a&=\frac{\mu_2^{\underline{\nu_2}}(\mu_1-(n+\frac{1}{2}))^{\underline{\nu_2}}}{(\nu_2!)(\nu_1-(n+\frac{1}{2}))^{\underline{\nu_2}}(n+\frac{1}{2})^{\underline{\nu_1-\nu_2}}} 
		\end{align*}
and
		\begin{align*}
			b=\sum_{i=0}^{\nu_1-\nu_2}{n+\frac{1}{2}\choose \nu_1-\nu_2-i}{n+\frac{1}{2}\choose i} (\mu_1-\nu_2-i)^{\underline{\nu_1-\nu_2-i}}(\mu_2-\nu_2)^{\underline{i}}.
		\end{align*}
\end{theorem}

	Even though the formula obtained for eigenvalues is typical of Capelli operators, i.e. Knop-Sahi polynomials for a specific $\rho$, we note that the standard technique of verifying the vanishing conditions (used in \cite{SSS20}) does not apply. The reason is that for $\nu$ such that $|\nu|>2n+1$, there are not enough vanishing conditions to allow us to use the uniqueness of the Knop-Sahi polynomials. However, we circumvent this issue by reducing the formula for $D_\nu$ to one for another operator, for which there are enough vanishing conditions.

	The proof of Theorem \ref{CapEvalThm} needs Theorem \ref{MainIntroTheorem} below, which answers the abstract Capelli problem (Problem \ref{AbsCapProb}) for the action of $\frak{gosp}(1|2n)$ on $\C^{1|2n}$ affirmatively. Define
		\begin{align}\label{CZPowersUpTod}
			\PP_{k,n}&:=\lbrace \mu\in\PP\colon |\mu|=k\textrm{ and }\mu_2\leq n \rbrace\quad\textrm{and}\quad \Lambda_{d,n}:=\bigcup_{k=0}^d\PP_{k,n}.
		\end{align}
Let $C$ denote the Casimir operator of $\osp(1|2n)$. 

\begin{theorem}\label{MainIntroTheorem}
	Let $d\in\Z_{\geq 0}$ and $\nu\in \Lambda_{d,n}^*$. The Capelli operator $D_\nu$ can be expressed as a linear combination of elements from the image of $B_{d,n}$ under the map \eqref{AbsCapProb}, where
	\begin{align}\label{CzUnderD}
		B_{d,n}:=\lbrace C^{\mu_2}Z^{\mu_1-\mu_2}\colon \mu\in \Lambda_{d,n}\rbrace.
	\end{align}
\end{theorem}

	We remark that Theorem \ref{MainIntroTheorem} cannot be proven by the general technique of \cite{SSS20}.

%
%
\section{The matrices $M_d$ and $M_d'$.}
%
%

The strategy to prove Theorem \ref{MainIntroTheorem} is as follows. We define a family of square matrices denoted by $M_d$, for $d\in \Z_{\geq 0}$, with coefficients in the polynomial ring $\Z[x]$. For all $d\in \Z_{\geq 0}$ the matrix $M_d$ will have the following property: if $\det(M_d)(n)\neq 0$ then $D_\nu$, for all $\nu\in \Lambda_{d,n}^*$, can be written as a linear combination of images of elements of $\emph{B}_{d,n}$.

We equip $\Lambda_{d,n}$, from \eqref{CZPowersUpTod}, and $\Lambda_{d,n}^*$, from \eqref{IrredUptoD}, with the following total orderings. 

\begin{defin}\label{Orderings}
	For $\mu,\mu'\in \Lambda_{d,n}$, define
	\begin{align*}
		\mu\preceq \mu' \textrm{ if and only if } \mu_2<\mu_2'\textrm{ or }(\mu_2=\mu_2'\textrm{ and } \mu_1\leq \mu_1')
	\end{align*}
and for $\nu,\nu'\in \Lambda_{d,n}^*$ define
	\begin{align*}
		\nu \leqslant \nu' \textrm{ if and only if } \left\lfloor \frac{\nu_1-\nu_2}{2}\right\rfloor< \left\lfloor \frac{\nu_1'-\nu_2'}{2} \right\rfloor\textrm{ or }\left( \left\lfloor \frac{\nu_1-\nu_2}{2}\right\rfloor=\left\lfloor \frac{\nu_1'-\nu_2'}{2} \right\rfloor\textrm{ and } |\nu|\leq |\nu'|\right).
	\end{align*}
\end{defin}

\begin{lemma}\label{prelemma}
	For any $d\in \Z_{\geq 0}$, we have $|\Lambda_{d,n}|=|\Lambda_{d,n}^*|.$
\end{lemma}

\begin{proof}
	This follows from the definition of $\Lambda_{d,n}^*$ and $\Lambda_{d,n}$ and the observation that for all $k\in \Z_{\geq 0}$, $|\PP_{k,n}|=\min\left(\left\lfloor \frac{k}{2}\right\rfloor,n  \right)+1=|\PP_{k,n}^*|.$
\end{proof}

\begin{rmrk}
	In particular, Lemma \ref{prelemma} implies $|B_{d,n}|=|\lbrace D_\nu \colon \nu\in \Lambda_{d,n}^*\rbrace|$.
\end{rmrk}

For $0\leq j\leq \min\left(\lfloor \frac{d}{2}\rfloor, n\right)$ define a family of subsets of $\Lambda_{d,n}$ and $\Lambda_{d,n}^*$ by
	\begin{align}\label{Subindexes}
		\Lambda_{d,n,j}&:=\lbrace \mu\in \Lambda_{d,n}\colon \mu_2\geq j\rbrace,\\
		\Lambda_{d,n,j}^*&:=\left\lbrace \nu\in \Lambda_{d,n}^*\colon \nu_2\leq \left\lfloor \frac{|\nu|}{2}\right\rfloor-j \right\rbrace,\nonumber
	\end{align}
and 
	\begin{align}\label{BoundarySubindexes}
		\partial\Lambda_{d,n,j}&:=\lbrace (\mu_1,j)\colon  \mu_1=j,j+1,\dotsc, d-j \rbrace,\\
		\partial\Lambda_{d,n,j}^*&:=\left\lbrace \left(k-\left\lfloor\frac{k}{2}\right\rfloor+j,\left\lfloor \frac{k}{2}\right\rfloor-j\right)\colon k=2j,2j+1,\dotsc, d \right\rbrace.\nonumber
	\end{align}
Observe that the following unions are disjoint:
	\begin{align}\label{LamdnjDecomp}
		\Lambda_{d,n,j}&=\Lambda_{d,n,j+1}\cup \partial\Lambda_{d,n,j}\\
		\Lambda_{d,n,j}^*&=\Lambda_{d,n,j+1}^*\cup \partial\Lambda_{d,n,j}^*.\nonumber
	\end{align}
The next corollary follows inductively from Equations \eqref{LamdnjDecomp}, Lemma \ref{prelemma} and the equalities $\Lambda_{d,n}=\Lambda_{d,n,0}$ and $\Lambda_{d,n}^*=\Lambda_{d,n,0}^*$.

\begin{corollary}
	We have $|\Lambda_{d,n,j}|=|\Lambda_{d,n,j}^*|$, for $0\leq j\leq \min\left(\lfloor \frac{d}{2}\rfloor, n\right)$.
\end{corollary}

\begin{defin}\label{PolyEntries}
For $\mu\in \Lambda_{d,n}$ and $\nu\in \Lambda_{d,n}^*$, set 
	\begin{align}\label{EllCoeff}
		\ell_{\nu}:=\min (\nu_1-\nu_2,2n+1-(\nu_1-\nu_2))
	\end{align}
and set
	\begin{align}\label{LambdaCoeff}
		\lambda_{\mu,\nu}(x)&:=((2x+1)\ell_{\nu}-\ell_{\nu}^2)^{\mu_2}(\nu_1+\nu_2)^{\mu_1-\mu_2}\in \Z[x].
	\end{align}
\end{defin}

\begin{defin}
	Let $M_d$ and $M_d'$ be square matrices defined by
		\begin{align}\label{Mdmatrix}
			M_d:=[\lambda_{\mu,\nu}]_{\substack{\mu\in\Lambda_{d,n}\\ \nu\in \Lambda_{d,n}^*}}
		\end{align}
	and
		\begin{align}\label{MdScalarMatrix}
			M_d':=[(2\ell_{\nu})^{\mu_2}(\nu_1+\nu_2)^{\mu_1-\mu_2} ]_{\substack{\mu\in\Lambda_{d,n}\\ \nu\in \Lambda_{d,n}^*}}.
		\end{align}
	The rows and columns of $M_d$ and $M_d'$ are ordered according to the total orderings given in Definition \ref{Orderings}.
\end{defin}

Notice that the entries of $M_d'$ are the leading coefficients of the polynomials $\lambda_{\mu,\nu}$.

\begin{rmrk}
	For $\nu\in \Lambda_n^*$, the Casimir operator $C$ acts on $V_\nu$ by the scalar $(2n+1)\ell_\nu-\ell_{\nu}^2.$ Furthermore, for any $\mu\in \PP$ the central element $C^{\mu_2}Z^{\mu_1-\mu_2}$ of  $\frak{U}(\g)$ acts on $V_\nu$ by $\lambda_{\mu,\nu}(n).$
\end{rmrk}


\begin{lemma}\label{LeadingIsMatrx}
If $\det(M_d')\neq 0$, then $\det(M_d)=\det(M_d')p(x)$ where $p(x)\in \Q[x]$ is monic.
\end{lemma}

\begin{proof}
	Notice that $\deg(\lambda_{\mu,\nu}(x))=\mu_2$, for all $\mu\in \Lambda_{d,n}$ and all $\nu\in \Lambda_{d,n}^*$. Furthermore, any monomial in the Leibniz determinant formula of $M_d$ is a product of distinct row entries from $M_d$. That is, all the non-zero monomials in this formula have the same degree, which is equal to $\sum_{\mu\in \Lambda_{d,n}} \mu_2.$ Thus if $\det(M_d')\neq 0$ then $\det(M_d')$ is the leading coefficient of $\det(M_d)$.
\end{proof}

\begin{rmrk}\label{DegreeRmrk}
	By the proof of Lemma \ref{LeadingIsMatrx}, if $\det(M_d')\neq 0$ then $\deg(\det(M_d))= \sum_{\mu\in \Lambda_{d,n}} \mu_2.$
\end{rmrk}

The remainder of this section is dedicated to showing that $\det(M_d')\neq 0$. The final results are given in Theorem \ref{MdScalarThm}.

\begin{defin}
For $\nu\in \Lambda_{d,n}^*$ and $0\leq j\leq \lfloor\frac{|\nu|}{2}\rfloor$, define
	\begin{align*}
		\nu_{(j)}:=\left(|\nu|-\left\lfloor\frac{|\nu|}{2}\right\rfloor+j,\left\lfloor\frac{|\nu|}{2}\right\rfloor-j\right)\in \partial\Lambda_{d,n,j}^*.
	\end{align*}
\end{defin}

\begin{corollary}\label{NuSubjCorollary}
	Let $\nu\in \Lambda_{d,n}^*$, $0\leq j\leq \lfloor \frac{|\nu|}{2}\rfloor$, and set $\nu'=\nu_{(j)}$. Then $\nu_{(k)}'=\nu_{(k)}$ for $0\leq k\leq \lfloor\frac{|\nu|}{2}\rfloor$.
\end{corollary}

\begin{lemma}\label{distNudistEll}
	Let $\nu,\nu'\in \Lambda_{d,n}^*$, such that $\nu\neq \nu'$ and $|\nu|=|\nu'|$. Then $\ell_{\nu}\neq \ell_{\nu'}.$
\end{lemma}

\begin{proof}
Straightforward by contradiction and considering all the possibilities of $\ell_\nu$ and $\ell_{\nu'}$ according to  Equation \eqref{EllCoeff}.
\end{proof}

\begin{corollary}\label{EllvNotEllvj}
	Let $1\leq j\leq \min\left( \lfloor \frac{d}{2}\rfloor ,n\right)-1$, let $\nu\in \Lambda_{d,n,j+1}^*$ and set $\nu'=\nu_{(j)}$. Then $\ell_\nu\neq \ell_{\nu'}.$
\end{corollary}

\begin{proof}
	Notice that $\nu_2\leq \left\lfloor \frac{|\nu|}{2}\right\rfloor-(j+1)=\nu_{2}'-1.$ That is, $\nu\neq \nu'$. Furthermore, $|\nu|=|\nu'|$ and therefore Lemma \ref{distNudistEll} implies $\ell_\nu\neq \ell_{\nu'}$.
\end{proof}

\begin{defin}
	Let $0\leq j \leq \min\left(\lfloor \frac{d}{2}\rfloor, n\right)$. 
	\begin{enumerate}[(i)]
		\item Let $\mu\in \Lambda_{d,n,j}$ and $\nu\in\Lambda_{d,n,j}^*$. Define 
	\begin{align*}
		S_{\mu,\nu,j}&:=\begin{cases}
		\ell_\nu^{\mu_2}&\textrm{ if }j=0,\\
		1 &\textrm{ if }\mu_2=j,\\
		\sum_{t_0=j}^{\mu_2-1}\sum_{t_1=j-1}^{t_0-1}\dotsb\sum_{t_{j}=0}^{t_{j-1}-1} \ell_{\nu_{(0)}}^{\mu_2-1-t_0}\ell_{\nu_{(1)}}^{t_0-1-t_1}\dotsb \ell_{\nu_{(j-1)}}^{t_{j-1}-1-t_j}\ell_{\nu}^{t_j}&\textrm{ otherwise.}
		\end{cases}
	\end{align*}
	\item Define the square matrix
		\begin{align*}
			N_j&:=[2^{\mu_2-j}(\nu_1+\nu_2)^{\mu_1-\mu_2}S_{\mu,\nu,j}]_{\substack{\mu\in \Lambda_{d,n,j}\\
			\nu\in\Lambda_{d,n,j}^*}}.
		\end{align*}
	\end{enumerate}
\end{defin}

\begin{lemma}\label{StoJ1Lemma}
	Let $0\leq j\leq \min\left(\lfloor\frac{d}{2}\rfloor, n\right)-1$, $\nu\in \Lambda_{d,n,j}$, and $\nu\in \Lambda_{d,n,j}^*$. Then 
	\begin{align*}
		S_{\mu,\nu,j}-S_{\mu,\nu_{(j)},j}=(\ell_\nu-\ell_{\nu_{(j)}})S_{\mu,\nu,j+1}.
	\end{align*}
\end{lemma}

\begin{proof}
	Using Corollary \ref{NuSubjCorollary} we have
		\begin{align*}
			S_{\mu,\nu,j}-S_{\mu,\nu_{(j)},j}&=\sum_{t_0=j+1}^{\mu_2-1}\sum_{t_1=j}^{t_0-1}\dotsb\sum_{t_{j}=1}^{t_{j-1}-1} \ell_{\nu_{(0)}}^{\mu_2-1-t_0}\ell_{\nu_{(1)}}^{t_0-1-t_1}\dotsb \ell_{\nu_{(j-1)}}^{t_{j-1}-1-t_j}(\ell_{\nu}^{t_j}-\ell_{\nu_{(j)}}^{t_j}).
		\end{align*}
	Now apply the identity 
		\begin{align*}
			(\ell_{\nu}^{t_j}-\ell_{\nu_{(j)}}^{t_j})=(\ell_{\nu}-\ell_{\nu_{(j)}})\left( \sum_{t_{j+1}=0}^{t_j-1}\ell_{\nu_{(j)}}^{t_j-1-t_{j+1}}\ell_{\nu}^{t_{j+1}} \right)
		\end{align*}
	to obtain
		\begin{align*}
		S_{\mu,\nu,j}-S_{\mu,\nu_{(j)},j}&=(\ell_{\nu}-\ell_{\nu_{(j)}})\sum_{t_0=j+1}^{\mu_2-1}\sum_{t_1=j}^{t_0-1}\dotsb\sum_{t_{j}=1}^{t_{j-1}-1}\sum_{t_{j+1}=0}^{t_j-1} \ell_{\nu_{(0)}}^{\mu_2-1-t_0}\ell_{\nu_{(1)}}^{t_0-1-t_1}\dotsb \ell_{\nu_{(j-1)}}^{t_{j-1}-1-t_j}\ell_{\nu_{(j)}}^{t_j-1-t_{j+1}}\ell_{\nu}^{t_{j+1}}\\
		&=(\ell_\nu-\ell_{\nu_{(j)}})S_{\mu,\nu,j+1}.\qedhere
		\end{align*}
\end{proof}

For the next lemma, recall the Vandermonde matrix defined for $\alpha_1,\dotsc,\alpha_s\in \C$ by
	\begin{align*}
		V(\alpha_1,\dotsc,\alpha_s)&:=[\alpha_j^{i-1}]_{1\leq i,j\leq s},
	\end{align*}
where $0^0:=1$. Indeed, $\det(V(\alpha_1,\dotsc,\alpha_s))=\prod_{1\leq j<i\leq s}(\alpha_i-\alpha_j)$.

\begin{lemma}\label{Njlemma}
	For $0\leq j\leq \min\left(\lfloor \frac{d}{2}\rfloor, n\right)-1$,
		\begin{align*}
			\det(N_j)&=\det(V(2j,2j+1,\dotsc, d))\det(N_{j+1})\left( \prod_{\nu\in \Lambda_{d,n,j+1}^*} 2\left( \ell_\nu-\ell_{\nu_{(j)}} \right) \right).
		\end{align*}

\end{lemma}

\begin{proof}
	Fix $0\leq j\leq \min\left( \lfloor \frac{d}{2}\rfloor, n  \right)-1$. Observe that $N_j$ has a submatrix 	
	\[
		[(\nu_1+\nu_2)^{\mu_1-j}]_{\substack{\mu \in \partial\Lambda_{d,n,j}\\ \nu\in \partial\Lambda_{d,n,j}^*}}=V(2j,2j+1,\dotsc, d).
	\]
	Therefore for each $\nu\in \Lambda_{d,n,j+1}^*$ we have $\nu_{(j)}\in \partial\Lambda_{d,n,j}^*$ and when we subtract the column indexed by $\nu_{(j)}$ from the column indexed by $\nu$, the entries at the $(\mu,\nu)$ position, where $\mu \in \partial\Lambda_{d,n,j}$, will vanish. The matrix obtained from $N_j$ after the above column operations is block lower triangular of the form $\begin{bmatrix} *&0\\ *&*\end{bmatrix}$. The determinant of the latter matrix is equal to $\det(N_j)$ and hence
		\begin{align*}
			\det(N_j)&=\det(V(2j,2j+1,\dotsc, d))\det\left( [2^{\mu_2-j}(\nu_1+\nu_2)^{\mu_1-\mu_2}(S_{\mu,\nu,j}-S_{\mu,\nu_{(j)},j})]_{\substack{\mu\in \Lambda_{d,n,j+1}\\ \nu\in \Lambda_{d,n,j+1}^*}} \right).
		\end{align*}
	Applying Lemma \ref{StoJ1Lemma} to the columns of the second factor on the right hand side 
	 
		\begin{align*}
			\det(N_j)&=\det(V(2j,2j+1,\dotsc, d))\det\left( [2^{\mu_2-j}(\nu_1+\nu_2)^{\mu_1-\mu_2}(\ell_\nu-\ell_{\nu_{(j)}})S_{\mu,\nu,j+1}]_{\substack{\mu\in \Lambda_{d,n,j+1}\\ \nu\in \Lambda_{d,n,j+1}^*}} \right)\\
			&=\det(V(2j,2j+1,\dotsc, d))\det(N_{j+1})\left( \prod_{\nu\in \Lambda_{d,n,j+1}^*} 2\left( \ell_\nu-\ell_{\nu_{(j)}} \right) \right).\qedhere
		\end{align*}
\end{proof}

\begin{theorem}\label{MdScalarThm}
	Let $M_d'$ be as defined in Equation \eqref{MdScalarMatrix}. Then $\det(M_d')=ab$ where
		\begin{align*}
			a= \prod_{j=1}^{\min\left(\lfloor\frac{d}{2}\rfloor,n\right)-1}\det(V(2j,2j+1,\dotsc, d))\left( \prod_{\nu\in \Lambda_{d,n,j+1}^*} 2\left( \ell_\nu-\ell_{\nu_{(j)}} \right) \right) 
		\end{align*}
	and
		\begin{align*}
			b=\det\left(V\left(2\min\left(\left\lfloor\frac{d}{2}\right\rfloor,n\right),2\min\left(\left\lfloor\frac{d}{2}\right\rfloor,n\right)+1,\dotsc,d\right)\right).
		\end{align*}
\end{theorem}

\begin{proof}
	This follows from the fact that $N_0=M_d'$, applying Lemma \ref{Njlemma} $\min(\lfloor\frac{d}{2}\rfloor,n)-1$ times and observing that
		\begin{align*}
			\det\left(N_{\min\left(\lfloor\frac{d}{2}\rfloor,n\right)}\right) &=\det\left(V\left(2\min\left(\left\lfloor\frac{d}{2}\right\rfloor,n\right),2\min\left(\left\lfloor\frac{d}{2}\right\rfloor,n\right)+1,\dotsc,d\right)\right).\qedhere
		\end{align*}
\end{proof}

%
%
\section{Factorization of $\det(M_d)$}
%
%

Theorem \ref{MdScalarThm} implies that $\det(M_d')\neq 0$.  Hence by Lemma \ref{LeadingIsMatrx} there exists a monic polynomial $p(x)\in \Q[x]$ such that
	\begin{align*}
		\det(M_d)&=\det(M_d')p(x).
	\end{align*}
This section is dedicated to factoring $p(x)$ into linear terms. The final results are given in Theorem \ref{MainTheorem}.

\begin{lemma}\label{FactorLambda}
	Let $\mu\in \Lambda_{d,n}$ and let $\nu,\nu'\in \Lambda_{d,n}^*$, such that $\nu\neq \nu'$ and $|\nu|=|\nu'|$. Then the polynomial $\lambda_{\mu,\nu}-\lambda_{\mu,\nu'}$	is divisible by $\left( x-\frac{\ell_{\nu}+\ell_{\nu'}-1}{2} \right)$, where $\lambda_{\mu,\nu}$ is defined in Equation \eqref{LambdaCoeff}.
\end{lemma}

\begin{proof}
	Let $\mu\in \Lambda_{d,n}$ and $\nu,\nu'\in \Lambda_{d,n}^*$, such that $|\nu|=|\nu'|$. Then $\nu_1+\nu_2=\nu_1'+\nu_2'$ and
		\begin{align*}
			\lambda_{\mu,\nu}-\lambda_{\mu,\nu'}&=(\nu_1+\nu_2)^{\mu_1-\mu_2}\left(((2x+1)\ell_\nu-\ell_{\nu}^2)^{\mu_2}-((2x+1)\ell_{\nu'}-\ell_{\nu'}^2)^{\mu_2}  \right)\\
			&=(\nu_1+\nu_2)^{\mu_1-\mu_2}(2(\ell_\nu-\ell_{\nu'}))\left( x-\frac{\ell_{\nu}+\ell_{\nu'}-1}{2} \right)\\
			&\hspace{1in}\left( \sum_{t=0}^{\mu_2-1}((2x+1)\ell_\nu-\ell_{\nu}^2)^{\mu_2-1-t}((2x+1)\ell_{\nu'}-\ell_{\nu'}^2)^{t} \right).\qedhere
		\end{align*}
\end{proof}

The next lemma will justify that Definition \ref{Matrixds} is valid.

\begin{lemma}\label{EllsAdded}
	Let $\nu,\nu'\in \Lambda_{d,n}^*$, such that $\nu\neq \nu'$ and $|\nu|=|\nu'|$. Then $$\ell_\nu+\ell_{\nu'}-1\in \lbrace 0,1,\dotsc, 2n-2 \rbrace.$$
\end{lemma}

\begin{proof}
	Let $0\leq k\leq d$ and $\nu,\nu'\in \Lambda_{d,n}^*$, such that $k=|\nu|=|\nu'|$. This proof is divided into three cases:
		\begin{itemize}
			\item If $k\leq n$, then $\ell_\nu,\ell_{\nu'}\in \lbrace k-2j\colon j=0,1,\dotsc,\lfloor\frac{k}{2}\rfloor \rbrace$. Lemma \ref{distNudistEll} implies that $\ell_\nu\neq \ell_{\nu'}$ and therefore 
				\begin{align*}
					\ell_\nu+\ell_{\nu'}-1\in \left\lbrace  2k-2j-1\colon j=1,2,\dotsc, 2\left\lfloor \frac{k}{2}\right\rfloor-1 \right\rbrace.
				\end{align*}
			\item If $n+1\leq k\leq 2n+1$, then $\ell_\nu$ and $\ell_{\nu'}$ take distinct values in the sets
			\begin{align}
				\left\lbrace k-2j\colon j=k-\left\lfloor \frac{n+k}{2}\right\rfloor,\dotsc,\left\lfloor\frac{k}{2}\right\rfloor \right\rbrace\label{LemmaCase2.1}\\
				\left\lbrace 2n+1-k+2j\colon j=0,\dotsc, k-\left\lfloor \frac{n+k}{2}\right\rfloor-1 \right\rbrace.\label{LemmaCase2.2}
			\end{align}
				\begin{itemize}
					\item If $\ell_\nu,\ell_{\nu'}$ are both in \eqref{LemmaCase2.1} then
						\begin{align*}
							\ell_\nu+\ell_{\nu'}-1\in \left\lbrace 2k-2j-1\colon j=2k-2\left\lfloor \frac{k}{2}\right\rfloor+1,\dotsc, 2\left\lfloor\frac{k}{2}\right\rfloor-1 \right\rbrace.
						\end{align*}
					\item For there to be atleast two distinct $\ell_\nu,\ell_{\nu'}$ in \eqref{LemmaCase2.2} then $k\geq n+3$. In which case we have
						\begin{align*}
							\ell_\nu+\ell_{\nu'}-1\in \left\lbrace 4n-2k+2j+1\colon j=1,2,\dotsc, 2k+2\left\lfloor \frac{n+1-k}{2}\right\rfloor-2n-3 \right\rbrace.
						\end{align*}
					\item If $\ell_\nu$ is in \eqref{LemmaCase2.1} and $\ell_{\nu'}$ in \eqref{LemmaCase2.2} then
						\begin{align*}
							\ell_\nu+\ell_{\nu'}-1\in \left\lbrace 2n-2j\colon j=1,2,\dotsc,\left\lfloor \frac{k}{2}\right\rfloor \right\rbrace.
						\end{align*}
				\end{itemize}
			\item If $k\geq 2n+2$ then $\ell_\nu$ and $\ell_{\nu'}$ take values from $\lbrace 0,1,\dotsc, n\rbrace$. It then follows for distinct $\ell_\nu$ and $\ell_{\nu'}$ that
				\begin{align*}
					\ell_\nu+\ell_{\nu'}&-1\in \lbrace 0,1,\dotsc, 2n-2\rbrace.\qedhere
				\end{align*}
		\end{itemize}
\end{proof}

For $s\in \lbrace 0,1,\dotsc, 2n-2\rbrace$ and $\nu\in \Lambda_{d,n}^*$ there is at most one $\nu'\in \Lambda_{d,n}^*$ such that $\ell_{\nu'}<\ell_{\nu}$, $|\nu|=|\nu'|$, and $\ell_\nu+\ell_{\nu'}-1=s$. Thus we can define the following matrix.

\begin{defin}\label{Matrixds}
	For $s=0,1,\dotsc, 2n-2$, define the square matrix $$M_{d,s}:=[\lambda_{\mu,\nu,s}]_{\substack{\mu\in\Lambda_{d,n}\\ \nu\in \Lambda_{d,n}^*}}$$ where $\lambda_{\mu,\nu,s}\in \Z[x]$ is defined by
		\begin{align*}
			\lambda_{\mu,\nu,s}(x)&:=\begin{cases}
				\lambda_{\mu,\nu}-\lambda_{\mu,\nu'} &\textrm{if there is}~ \nu'\in \Lambda_{d,n}^*,~ |\nu|=|\nu'|,~ \ell_{\nu'}<\ell_\nu,~ s=\ell_\nu+\ell_{\nu'}-1,\\
				\lambda_{\mu,\nu}&\textrm{otherwise.}
			\end{cases}
		\end{align*}
\end{defin}

\begin{lemma}\label{detMatrixds}
	Let $s\in \lbrace 0,1,\dotsc, 2n-2\rbrace$. Then $\det(M_d)=\det(M_{d,s}).$ Furthermore, $\det(M_d)$ is divisible, in $\Q[x]$, by $\left(x-\frac{s}{2}\right)^{f(d,s)},$
	where
	\begin{align}\label{gks}
		f(d,s)=\big|\lbrace (\nu,\nu')\in \Lambda_{d,n}^*\times \Lambda_{d,n}^*\colon |\nu|=|\nu'|, \ell_{\nu'}<\ell_{\nu}, \ell_\nu+\ell_{\nu'}-1=s \rbrace\big|
	\end{align}
\end{lemma}

\begin{proof}
	Notice that through column operations $M_{d,s}$ is obtained from $M_d$. In particular, for $\nu\in \Lambda_{d,n}^*$, column $\nu$ of $M_{d,s}$ is either column $\nu$ of $M_d$ or is column $\nu$ minus column $\nu'$ of $M_d$, where $|\nu|=|\nu'|$, $\ell_\nu>\ell_{\nu'}$, and $\ell_\nu+\ell_{\nu'}-1=s$. These operations can be done simultaneously by starting with the largest $\ell_\nu$ and performing the necessary column operations. This leaves all other columns untouched. Thus we continue to the second largest $\ell_\nu$ and likewise for all $\ell_\nu$. It follows that $\det(M_d)=\det(M_{d,s}).$ Next take any $$(\nu,\nu')\in\lbrace (\nu,\nu')\in \Lambda_{d,n}^*\times \Lambda_{d,n}^*\colon |\nu|=|\nu'|, \ell_{\nu'}<\ell_{\nu}, \ell_\nu+\ell_{\nu'}-1=s \rbrace.$$
	Lemma \ref{FactorLambda} implies that the column $\nu$ of $M_{d,s}$ is divisible by $(x-\frac{s}{2})$. Thus $(x-\frac{s}{2})$ divides $\det(M_{d,s})=\det(M_d)$ with multiplicity
		\begin{align*}
			f(d,s)&=\big|\lbrace (\nu,\nu')\in \Lambda_{d,n}^*\times \Lambda_{d,n}^*\colon |\nu|=|\nu'|, \ell_{\nu'}<\ell_{\nu}, \ell_\nu+\ell_{\nu'}-1=s \rbrace\big|.\qedhere
		\end{align*}
\end{proof}

\begin{rmrk}\label{gks}
	Notice that $f(d,s)=\sum_{k=0}^d g_k(s)$ where 
		\begin{align*}
			g_k(s):=\big|\lbrace (\nu,\nu')\in \mathbb{P}_{k,n}^*\times \mathbb{P}_{k,n}^*\colon \ell_{\nu'}<\ell_{\nu}, \ell_\nu+\ell_{\nu'}-1=s \rbrace\big|.
		\end{align*}
%
\end{rmrk}

\details{ In the next Lemma we prove a stronger version of Lemma \ref{RootMultink}. Indeed, the proof of Lemma \ref{RootMult} depends on Lemma \ref{RootMultink}.

\begin{lemma}\label{RootMult}
	For $n,d\in \Z_{\geq 0}$ and $s=0,1,\dotsc, 2n-2$, the term
		\begin{align*}
			\left( x-\frac{s}{2} \right)^{f(d,s)}
		\end{align*}
	factors from $\det(M_d)$, where $f(d,s)$ is defined:
		\begin{itemize}
			\item for $0\leq d\leq n$
				\begin{align*}
					f(d,s)&=\begin{cases}
						f(s-1,s)+\left\lfloor \frac{s+2}{4} \right\rfloor\left(  \left\lfloor \frac{d+s}{2} \right\rfloor-s+1\right)+\left\lfloor \frac{s+4}{4} \right\rfloor\left(\frac{2\lfloor \frac{d}{2}\rfloor-s-1}{2}\right)&~\textrm{if }s\textrm{ odd}, 1\leq s \leq d-2,\\
					\frac{1}{2}\left( \left\lfloor \frac{d-\lfloor\frac{s}{2}\rfloor}{2}\right\rfloor\left( \left\lfloor \frac{d-\lfloor\frac{s}{2}\rfloor}{2}\right\rfloor+1\right)+\left\lfloor \frac{d-\lfloor\frac{s}{2}\rfloor-1}{2}\right\rfloor\left(\left\lfloor \frac{d-\lfloor\frac{s}{2}\rfloor-1}{2}\right\rfloor+1\right)\right)&~\textrm{if $s$ odd},~ d-1\leq	s\leq 2d-3,\\
					0&\textrm{ otherwise.}
					\end{cases}
				\end{align*}
			\item for $n+1\leq d\leq 2n+1$ then 
				\begin{itemize}
					\item if $s$ is odd then $f(d,s)$ is given by
				\begin{align*}
					f(d,s)&=\begin{cases}
						h_0(d,s)&\textrm{ if }1\leq s\leq 4n-2d+1,\\
						h_0(d,s)+h_1(d,s)&\textrm{ if }4n-2d+3\leq s\leq 2n-d-1,\\
						h_0(d,s)+h_1(2n-s,s)+h_2(d,s)&\textrm{ if }2n-d+1\leq s\leq 2n-2
					\end{cases}
 				\end{align*}
		where
				\begin{align*}
					\hspace{-0.5in}h_0(d,s)&=f(n,s)+\left( d-\left\lfloor \frac{n+d-2}{2}\right\rfloor\right)g_{2\left\lfloor \frac{n+d}{2}\right\rfloor -d}(s)+\left( d-\left\lfloor \frac{n+d-1}{2}\right\rfloor \right) g_{2\left\lfloor \frac{n+d+1}{2}\right\rfloor-d-1}(s)\\
					\hspace{-0.5in}h_1(d,s)&=\frac{1}{2}\left(\left\lfloor \frac{d-2n+\lfloor \frac{s}{2}\rfloor+1}{2} \right\rfloor \left(\left\lfloor \frac{d-2n+\lfloor \frac{s}{2}\rfloor+1}{2} \right\rfloor+1 \right)+\left\lfloor \frac{d-2n+\lfloor \frac{s}{2}\rfloor}{2} \right\rfloor\left(\left\lfloor \frac{d-2n+\lfloor \frac{s}{2}\rfloor}{2} \right\rfloor+1  \right) \right)\\
				\hspace{-0.5in}h_2(d,s)&=\left\lfloor \frac{2n-s}{4}\right\rfloor \left( \left\lfloor \frac{d-n}{2}\right\rfloor -n+\left\lfloor \frac{s}{2}\right\rfloor +1 \right)+\left\lfloor \frac{n-\lfloor \frac{s}{2}\rfloor}{2}\right\rfloor \left( \left\lfloor \frac{d-n+s}{2}\right\rfloor -n+\left\lfloor\frac{s}{2}\right\rfloor\right)
				\end{align*}
					\item if $s$ is even then $f(d,s)$ is given by
				\begin{align*}
					\begin{cases}
						\frac{1}{2}\left(\left\lfloor \frac{d-2n+s+2}{2}\right\rfloor\left( \left\lfloor \frac{d-2n+s+2}{2}\right\rfloor+1 \right)+\left\lfloor \frac{d-2n+s+1}{2}\right\rfloor\left( \left\lfloor \frac{d-2n+s+1}{2}\right\rfloor+1 \right)\right)&\textrm{ if }2(n-\lfloor \frac{d}{s}\rfloor)\leq s\leq 2\lfloor \frac{n-2}{2}\rfloor,\\
						\frac{1}{2}\left(\left\lfloor \frac{d-2n+s+2}{2}\right\rfloor\left( \left\lfloor \frac{d-2n+s+2}{2}\right\rfloor+1 \right)+\left\lfloor \frac{d-2n+s+1}{2}\right\rfloor\left( \left\lfloor \frac{d-2n+s+1}{2}\right\rfloor+1 \right)  \right)&\textrm{ if }2\lfloor\frac{n}{2}\rfloor\leq s\leq 2(n-d+\lfloor \frac{n+d}{2}\rfloor),\\
					\left(n-\frac{s}{2}\right)^2+\left(n-\frac{s}{2}\right)( d-3n+s+1)&\textrm{ if }2(n-d+\lfloor \frac{n+d}{2}\rfloor+1)\leq s\leq 2n-2.
					\end{cases}
				\end{align*}
				\end{itemize}
			\item for $d\geq 2n+1$
				\begin{align*}
					f(d,s)&=f(2n+1,s)+(d-2n-1)g_{2n+1}(s).
				\end{align*}
		\end{itemize}
\end{lemma}

\begin{proof}
	Fix $s=0,1,\dotsc, 2n-2$. Lemma \ref{RootMultink} implies that for all $2\leq k\leq d$ the term $\left( x-\frac{s}{2} \right)^{g_k(s)}$  can be factors from $\det(M_d)$. But in the proof of Lemma \ref{RootMultink} it is seen that this factor results only from operations on columns indexed by $\nu\in \Lambda_{d,n}^*$ with $\nu_1+\nu_2=k$. Thus we can simultaneously apply Lemma \ref{RootMultink} for all $0\leq k\leq d$ to conclude that $\left(x-\frac{s}{2}\right)$ can be factored with multiplicity
		\begin{align}\label{SumToMulti}
			\sum_{k=2}^d g_k(s).
		\end{align}
We will show that $f(d,s)$ is exactly Equation \eqref{SumToMulti}. We do this in cases:
	\begin{itemize}
		\item When $0\leq d\leq n$ we have that $g_k(s)\neq 0$ if $k\geq \lfloor \frac{s}{2}\rfloor+2$, $s$ is odd, and $k-2\lfloor \frac{k}{2}\rfloor\leq s\leq 2k-3$. We then break up this case into the following subcases:
			\begin{itemize}
				\item If $d-1\leq s\leq 2d-3$ then
					\begin{align*}
						\sum_{k=0}^d g_k(s)&=\sum_{k=\lfloor \frac{s}{2}\rfloor +2}^{d-2} \left\lfloor \frac{2k+1-s}{4} \right\rfloor+g_{d-1}(s)+g_d(s)\\
						&=\sum_{i=2}^{d-2-\lfloor\frac{s}{2}\rfloor}\left\lfloor \frac{i}{2}\right\rfloor +g_{d-1}(s)+g_d(s)\\
						&=f(d,s).
					\end{align*}
	The last equality can be confirmed by directly checking the possible values for $g_{d-1}(s)$ and $g_{d}(s)$.
				\item If $1\leq s\leq d-2$ then we have
					\begin{align*}
						\sum_{k=0}^d g_k(s)&=\sum_{k=\lfloor\frac{s}{2}\rfloor+2}^{s-1} g_k(s)+\sum_{k=s}^{d} g_k(s)\\
						&=\sum_{k=\lfloor\frac{s}{2}\rfloor+2}^{s-1}\left\lfloor \frac{2k+1-s}{4} \right\rfloor+\sum_{k=s}^d\left\lfloor \frac{s-2(k-2\lfloor\frac{k}{2}\rfloor -2)}{4} \right\rfloor\\
						&=\sum_{i=2}^{\lfloor\frac{s}{2}\rfloor}\left\lfloor\frac{i}{2}\right\rfloor+\sum_{k=s}^d\left\lfloor \frac{s-2k+4\lfloor\frac{k}{2}\rfloor +2)}{4} \right\rfloor\\
						&=f(s-1,s)+\left\lfloor \frac{s+2}{4} \right\rfloor\left(  \left\lfloor \frac{d+s}{2} \right\rfloor-s+1\right)+\left\lfloor \frac{s+4}{4} \right\rfloor\left(\frac{2\lfloor \frac{d}{2}\rfloor-s-1}{2}\right)=f(d,s).
					\end{align*}
			\end{itemize}
		\item When $n+1\leq d\leq 2n+1$ we consider the cases for $s$ odd and even. For $s$ odd:
			\begin{itemize}
				 \item If $1\leq s\leq 4n-2d+1$, then $g_k(s)=g_{2\lfloor \frac{n+k}{2}\rfloor -k}(s)$ for all $n+1\leq k\leq d$. Therefore
				 	\begin{align*}
				 		\sum_{k=0}^d g_k(s)&=f(n,s)+\sum_{k=n+1}^dg_{2\lfloor \frac{n+k}{2}\rfloor -k}(s)\\
				 		&=f(n,s)+\left( d-\left\lfloor \frac{n+d}{2}\right\rfloor +1 \right)g_{2\lfloor \frac{n+d}{2}\rfloor -d}(s)\\
				 		&\hspace{2in}+\left( d-\left\lfloor \frac{n+d-1}{2} \right\rfloor \right)g_{2\lfloor \frac{n+d-1}{2}\rfloor -d+1}(s)\\
				 		&=f(d,s).
				 	\end{align*}
				 \item If $4n-2d+3\leq s\leq 2n+2\lfloor \frac{n-d-1}{2}\rfloor +1$, then $g_k(s)=g_{2\lfloor \frac{n+k}{2}\rfloor -k}(s)$ for all $n+1\leq k\leq 2n-\lfloor\frac{s}{2}\rfloor$ and $g_k(s)=g_{2\lfloor \frac{n+k}{2}\rfloor -k}(s)+\left\lfloor \frac{s-4n+2k+1}{4} \right\rfloor$ for all $2n-\lfloor\frac{s}{2}\rfloor+1\leq k\leq d$. Therefore
				 	\begin{align*}
				 		\sum_{k=0}^d g_k(s)&=f(n,s) +\sum_{k=n+1}^d g_{2\lfloor \frac{n+k}{2}\rfloor-k}(s)+\sum_{k=2n-\lfloor \frac{s}{2}\rfloor+1}^d \left\lfloor \frac{s-4n+2k+1}{4} \right\rfloor\\
				 		&=h_0(d,s)+\sum_{i=1}^{d-2n+\lfloor\frac{s}{2}\rfloor}\left\lfloor \frac{i+1}{2}\right\rfloor\\
				 		&=h_0(d,s)+h_1(d,s)=f(d,s).
				 	\end{align*}
				\item If $2n+2\lfloor \frac{n-d-1}{2}\rfloor +3\leq s\leq 2n-3$, then 
					\begin{align*}
						g_k(s)&=\begin{cases}
							g_{2\lfloor \frac{n+k}{2}\rfloor -k}(s) &\textrm{ for }n+1\leq k\leq 2n-\lfloor\frac{s}{2}\rfloor,\\
							g_{2\lfloor \frac{n+k}{2}\rfloor -k}(s)+\left\lfloor \frac{s-4n+2k+1}{4} \right\rfloor&\textrm{ for }2n-\lfloor\frac{s}{2}\rfloor+1\leq k\leq 3n-s,\\
							g_{2\lfloor \frac{n+k}{2}\rfloor -k}(s)+\left\lfloor \frac{4\lfloor\frac{n+1-k}{2}\rfloor +2k-s}{4} \right\rfloor&\textrm{ for all }3n-s+1\leq k\leq d.
						\end{cases}
					\end{align*}
				Therefore
					\begin{align*}
						\sum_{k=0}^d g_k(s)&=f(n,s)+\sum_{k=n+1}^d g_{2\lfloor \frac{n+k}{2}\rfloor-k}(s)+\sum_{k=2n-\lfloor \frac{s}{2}\rfloor+1}^{3n-s} \left\lfloor \frac{s-4n+2k+1}{4} \right\rfloor\\
						&\hspace{1in}+\sum_{k=3n-s+1}^d\left\lfloor \frac{4\lfloor \frac{n+1-k}{2}\rfloor+2k-s}{4}\right\rfloor\\
						&=h_0(d,s)+h_1(3n-s,s)+\sum_{i=n-\lfloor \frac{s}{2}\rfloor}^{\lfloor \frac{d-n}{2}\rfloor} \left\lfloor \frac{2n-s}{4}\right\rfloor +\sum_{i=n-\lfloor\frac{s}{2}\rfloor}^{\lfloor\frac{d-n+1}{2}\rfloor-1}\left\lfloor \frac{2n-s+2}{4}\right\rfloor\\
						&=h_0(d,s)+h_1(3n-s,s)+h_2(d,s)=f(d,s).
					\end{align*}
			\end{itemize}
			For $s$ even:
				\begin{itemize}
					\item if $2(n-\lfloor\frac{d}{2}\rfloor)\leq s\leq 2\lfloor\frac{n-2}{2}\rfloor$ then
						\begin{align*}
							g_k(s)&=\begin{cases}
								0& \textrm{ if }0\leq k\leq 2n-s-1\\
								\frac{s}{2}-n+\lfloor \frac{k}{2}\rfloor +1&\textrm{ if }2n-s\leq k\leq d.
							\end{cases}
						\end{align*}
					Therefore
						\begin{align*}
							\sum_{k=0}^d g_k(s)&=\sum_{k=2n-s}^d \frac{s}{2}-n+\lfloor \frac{k}{2}\rfloor +1\\
							&=\frac{1}{2}\left(\left\lfloor \frac{d-2n+s+2}{2}\right\rfloor\left( \left\lfloor \frac{d-2n+s+2}{2}\right\rfloor+1 \right)\right.\\
							&\hspace{1in}\left.+\left\lfloor \frac{d-2n+s+1}{2}\right\rfloor\left( \left\lfloor \frac{d-2n+s+1}{2}\right\rfloor+1 \right)  \right)\\
							&=f(d,s).
						\end{align*}
					\item if $2\lfloor\frac{n}{2}\rfloor\leq s\leq 2(n-d+\lfloor \frac{n+d}{2}\rfloor)$ then 
						\begin{align*}
							g_k(s)=\begin{cases}
								0&\textrm{ if }0\leq k\leq n,\\
								k-\lfloor\frac{n+k}{2}\rfloor&\textrm{ if }n+1\leq k\leq d.
							\end{cases}
						\end{align*}
					Therefore
						\begin{align*}
							\sum_{k=0}^d g_k(s)&=\sum_{k=n+1}^d k-\left\lfloor \frac{n+k}{2}\right\rfloor\\
							&=\frac{1}{2}\left(\left\lfloor \frac{d-n+1}{2} \right\rfloor\left( \left\lfloor \frac{d-n+1}{2} \right\rfloor+1 \right)\right.\\
							&\hspace{1in}+\left.\left\lfloor \frac{d-n}{2} \right\rfloor\left( \left\lfloor \frac{d-n}{2} \right\rfloor+1 \right) \right)\\
							&=f(d,s).
						\end{align*}
					\item if $2(n-d+\lfloor \frac{n+d}{2}\rfloor+1)\leq s\leq 2n-2$ then
						\begin{align*}
							g_k(s)&=\begin{cases}
								0&\textrm{ if }0\leq k\leq n,\\
								k-\lfloor\frac{n+k}{2}\rfloor&\textrm{ if } n+1\leq k\leq 3n-s-1,\\
								n-\frac{s}{2}&\textrm{ if }3n-s\leq k\leq d.
							\end{cases}
						\end{align*}
					Therefore
						\begin{align*}
							\sum_{k=0}^d g_k(s)&=\sum_{k=n+1}^{3n-s-1} k-\left\lfloor \frac{n+k}{2}\right\rfloor+\sum_{k=3n-s}^d n-\frac{s}{2}\\
							&=\frac{1}{2}\left(\left\lfloor \frac{2n-s}{2} \right\rfloor\left( \left\lfloor \frac{2n-s}{2} \right\rfloor+1 \right)+\left\lfloor \frac{2n-s-1}{2} \right\rfloor\left( \left\lfloor \frac{2n-s-1}{2} \right\rfloor+1 \right) \right)\\
							&\hspace{1in}+\left(n-\frac{s}{2}\right)( d-3n+s+1)\\
							&=f(d,s)
						\end{align*}
				\end{itemize}
		\item For $d\geq 2n+2$ then it is straight forward to get
			\begin{align*}
				\sum_{k=2}^d g_k(s)&=\sum_{k=0}^{2n+1} g_k(s)+\sum_{k=2n+2}^d g_{2n+1}(s)\\
				&=f(2n+1,s)+ (d-2n-1)g_{2n+1}(s)\\
				&=f(d,s).\qedhere
			\end{align*}
	\end{itemize}
\end{proof}}

\begin{lemma}\label{SumEquality}
	Let $d\in \Z_{\geq 0}$. Then $\sum_{s=0}^{2n-2}f(d,s)=\sum_{\mu\in \Lambda_{d,n}} \mu_2.$
\end{lemma}

\begin{proof}
		First, we have the identity
			\begin{align*}
					\sum_{s=0}^{2n-2} g_k(s)&=\frac{1}{2} \min\left(\left\lfloor\frac{k}{2}\right\rfloor,n\right)\left( \min\left(\left\lfloor\frac{k}{2}\right\rfloor,n\right)+1\right),
			\end{align*}
		for $g_k(s)$ defined in Remark \ref{gks}. This identity follows from the fact that the left hand side is counting the cardinality of
			\begin{align*}
				\bigcup_{s=0}^{2n-2}\lbrace (\nu,\nu')\in \mathbb{P}_{k,n}^*\times \mathbb{P}_{k,n}^*\colon \ell_{\nu'}<\ell_\nu,~\ell_\nu+\ell_{\nu'}-1=s \rbrace=\lbrace (\nu,\nu')\in \mathbb{P}_{k,n}^*\times \mathbb{P}_{k,n}^*\colon \ell_{\nu'}<\ell_\nu \rbrace
			\end{align*}
and the fact that there are $\min\left(\left\lfloor\frac{k}{2}\right\rfloor,n\right)+1$ possible $\ell_\nu$. Therefore we have
			\begin{align*}
				\sum_{s=0}^{2n-2}f(d,s)&=\sum_{s=0}^{2n-2}\sum_{k=0}^d g_k(s)=\sum_{k=0}^d \sum_{s=0}^{2n-2} g_k(s)\\
				&=\sum_{k=0}^d\frac{1}{2} \min\left(\left\lfloor\frac{k}{2}\right\rfloor,n\right)\left( \min\left(\left\lfloor\frac{k}{2}\right\rfloor,n\right)+1\right).
			\end{align*}
	Next $\sum_{\mu\in \Lambda_{d,n}}\mu_2$ is calculated for the following two cases.
		\begin{itemize}
			\item For $0\leq d\leq 2n+1$,
				\begin{align*}
					\sum_{\mu\in \Lambda_{d,n}}\mu_2&=\sum_{\mu_1=0}^{\lfloor \frac{d}{2}\rfloor}\sum_{\mu_2=0}^{\mu_1}\mu_2+\sum_{\mu_1=\lfloor \frac{d}{2}\rfloor+1}^d\sum_{\mu_2=0}^{d-\mu_1}\mu_2\\
					&=\sum_{\mu_1=0}^{\lfloor \frac{d}{2}\rfloor}\frac{\mu_1(\mu_1+1)}{2}+\sum_{\mu_1=\lfloor \frac{d}{2}\rfloor+1}^d\frac{(d-\mu_1)(d-\mu_1+1)}{2}=\sum_{\mu_1=0}^d\frac{\left\lfloor\frac{\mu_1}{2}\right\rfloor(\left\lfloor\frac{\mu_1}{2}\right\rfloor+1)}{2}.
				\end{align*}
			\item For $d\geq 2n+2$ we have the following analogous calculation
				\begin{align*}
					\sum_{\mu\in \Lambda_{d,n}}\mu_2&=\sum_{\mu_1=0}^{n}\sum_{\mu_2=0}^{\mu_1}\mu_2+\sum_{\mu_1=n+1}^{d-n-1}\sum_{\mu_2=0}^n \mu_2+\sum_{\mu_1=d-n}^d\sum_{\mu_2=0}^{d-\mu_1}\mu_2\\
					&=\sum_{\mu_1=0}^{2n+1}\frac{\left\lfloor\frac{\mu_1}{2}\right\rfloor\left( \left\lfloor\frac{\mu_1}{2}\right\rfloor+1 \right)}{2}+(d-2n-1)\frac{n(n+1)}{2}.
				\end{align*}
		\end{itemize}
	For both of these cases, it is easy to verify that $\sum_{s=0}^{2n-2}f(d,s)=\sum_{\mu\in \Lambda_{d,n}} \mu_2.$
\end{proof}

Finally, we have all the tools needed to prove the main result of this section.

\begin{theorem}\label{MainTheorem}
	Let $d\in \Z_{\geq 0}$. Then
		\begin{align*}
			\det(M_d)&=\det(M_d')\prod_{s=0}^{2n-2} \left( x-\frac{s}{2}\right)^{f(d,s)},
		\end{align*}
	where $f(d,s)$ is given in Lemma \ref{detMatrixds}.
\end{theorem}

\begin{proof}
	Recall that Lemma \ref{LeadingIsMatrx} and Theorem \ref{MdScalarThm} imply that $\det(M_d)=\det(M_d')p(x),$ for some monic $p(x)\in \Q[x]$. Furthermore, Lemma \ref{detMatrixds} implies that $\left(x-\frac{s}{2}\right)^{f(d,s)}$ divides $p(x)$, for all $s=0,\dotsc, 2n-2$. Therefore there is some monic $q(x)\in \Q[x]$ such that
		\begin{align*}
			p(x)&=q(x)\prod_{s=0}^{2n-2} \left( x-\frac{s}{2}\right)^{f(d,s)}.
		\end{align*}
	In particular, $\deg(p(x))\geq \sum_{s=0}^{2n-2}f(d,s).$ But Remark \ref{DegreeRmrk} tells us that $\deg(p(x))= \sum_{\mu\in \Lambda_{d,n}}\mu_2$. It then follows from Lemma \ref{SumEquality} that $$\sum_{\mu\in \Lambda_{d,n}}\mu_2=\sum_{s=0}^{2n-2}f(d,s)\leq \deg(p(x))= \sum_{\mu\in \Lambda_{d,n}}\mu_2.$$
	That is, $q(x)=1$ and 
		\begin{align*}
			\det(M_d)&=\det(M_d')p(x)=\det(M_d')\prod_{s=0}^{2n-2} \left( x-\frac{s}{2}\right)^{f(d,s)}.\qedhere
		\end{align*}
\end{proof}

Thus we can prove Theorem \ref{MainIntroTheorem} and thereby answer Problem \ref{AbsCapProb} in the affirmative.\\

\noindent \emph{Proof of Theorem \ref{MainIntroTheorem}.} Let $d\in \Z_{\geq 0}$. Recall that the image of $C^{\mu_2}Z^{\mu_1-\mu_2}$ is a differential operator of order $|\mu|$ and therefore is an element of $\bigoplus_{k=0}^{|\mu|}\mathcal{PD}^{k}(\C^{1|2n})^\g$. Lemma \ref{prelemma} implies $|B_{d,n}|=|\lbrace D_\nu \colon |\nu|\leq d\rbrace|$. Therefore if $B_{d,n}$ is a linearly independent set then it is a basis of $\bigoplus_{k=0}^d\mathcal{PD}^k(\C^{1|2n})^\g$. 

Identify $C^{\mu_2}Z^{\mu_1-\mu_2}$ with its image and fix non-zero $v_\nu\in V_\nu $, for each $\nu\in \Lambda_{d,n}^*$. Let $a_\mu\in \C$ be such that 
	\begin{align}\label{linearcombCZ}
		\sum_{\mu\in \Lambda_{d,n}} a_\mu C^{\mu_2}Z^{\mu_1-\mu_2}=0.
	\end{align}
We evaluate \eqref{linearcombCZ} on each $v_\nu$. Since $C^{\mu_2}Z^{\mu_1-\mu_2}$ acts on $v_\nu$ by $\lambda_{\mu,\nu}(n)$ we have the following matrix relation
	\begin{align*}
		[a_\mu]_{\mu\in \Lambda_{d,n}}[\lambda_{\mu,\nu}(n)]_{\substack{\mu\in\Lambda_{d,n}\\ \nu\in \Lambda_{d,n}^*}}=0,
	\end{align*}
where $[a_\mu]_{\mu\in \Lambda_{d,n}}$ denotes a row vector. The matrix $[\lambda_{\mu,\nu}(n)]_{\substack{\mu\in\Lambda_{d,n}\\ \nu\in \Lambda_{d,n}^*}}$ is exactly $M_d$ with entries evaluated at $x=n$. Theorem \ref{MainTheorem} implies that $\det(M_d)(n)\neq 0$ and therefore $[\lambda_{\mu,\nu}(n)]_{\substack{\mu\in\Lambda_{d,n}\\ \nu\in \Lambda_{d,n}^*}}$ is invertible. In particular, $a_\mu=0$, for all $\mu\in \Lambda_{d,n}$ and hence $B_{d,n}$ is linearly independent.\qed

%
%
\section{Solution of the Capelli Eigenvalue Problem}
%
%

In the previous section we proved Theorem \ref{MainIntroTheorem} and thereby answered Problem \ref{AbsCapProb}. For the remainder of this article we will answer Problem \ref{CapEvalProb} by proving Theorem \ref{CapEvalThm}. We will need the following lemma, which can be found in \cite[Lemma 1]{DebEel09}.

\begin{lemma}\label{LapRRel}
	Let $w\in \mathcal{P}^k(V)$. Then
		\begin{align*}
			\nabla^2(R^{2t}w)&=2t(2k+1-2n+2(t-1))R^{2(t-1)}w+R^{2t}\nabla^2 w.
		\end{align*}
\end{lemma}

\begin{proof}
	By induction on $t$, using $[\nabla^2,R^2]=4E+2(1-2n)$ we have
		\begin{align*}
			\nabla^2(R^2 w)&=(4k+2(1-2n))w+R^2\nabla^2w.\qedhere
		\end{align*}
\end{proof}

\begin{lemma}\label{CapRecursion}
	Let $\lambda\in \Lambda_n^*$ such that $\lambda_2\geq 1$ and set $\nu=(\lambda_1-1,\lambda_2-1)$. The Capelli operators $D_\lambda$ and $D_\nu$ are related by
		\begin{align*}
			2\lambda_2(2\lambda_1-2n-1)D_\lambda&=R^2 D_\nu \nabla^2.
		\end{align*}
\end{lemma}

\begin{proof}
Let $\lbrace R^{2\nu_2}v_1,\dotsc, R^{2\nu_2}v_\ell\rbrace$, with $v_i$ harmonic, be a basis to $V_\nu$ and let $\lbrace v_1^*,\dotsc, v_\ell^*\rbrace$ be its dual basis. It follows that $D_\nu=\sum_{i=1}^\ell \left(R^{2\nu_2}v_i\right)v_i^*.$ Furthermore, $\lbrace R^{2\lambda_2}v_1,\dotsc,R^{2\lambda_2}v_\ell\rbrace$ is a basis to $V_{\lambda}$. Notice that Lemma \ref{LapRRel} implies that for all $1\leq i,j\leq \ell$ we have
	\begin{align*}
		v_i^*\nabla^2(R^{2\lambda_2}v_j)&=v_i^*\left( 2\lambda_2(2\lambda_1-2n-1)R^{2\nu_2}v_j+R^{2\lambda_2}\nabla^2v_j \right)\\
		&=2\lambda_2(2\lambda_1-2n-1)\delta_{i,j}.
	\end{align*}
Thus $\left\lbrace \frac{v_i^*\nabla^2}{2\lambda_2(2\lambda_2-2n-1)}\mid 1\leq i\leq \ell\right\rbrace$ is dual to $\lbrace R^{2\lambda_2}v_1,\dotsc,R^{2\lambda_2}v_\ell\rbrace$ and
	\begin{align*}
		D_\lambda&=\sum_{i=1}^\ell \frac{(R^{2\lambda_2}v_i)( v_i^*\nabla^2)}{2\lambda_2(2\lambda_1-2n-1)}=\frac{R^2D_\nu\nabla^2}{2\lambda_2(2\lambda_1-2n-1)}.\qedhere
	\end{align*}
\end{proof}

\begin{prop}\label{CapEvalRecursive}
	Let $\lambda,\mu\in \Lambda_n^*$ be such that $\lambda_2,\mu_2\geq 1$. Set $\nu=(\lambda_1-1,\lambda_2-1)$ and $\eta=(\mu_1-1,\mu_2-1)$. Then
		\begin{align*}
			c_{\lambda}(\mu)&=\frac{\mu_2(\mu_1-(n+\frac{1}{2}))}{\lambda_1(\lambda_1-(n+\frac{1}{2}))}c_{\nu}(\eta).
		\end{align*}
\end{prop}

\begin{proof}
	Take any $R^{2\mu_2}w\in V_\mu$, with $w$ harmonic, then using Lemma \ref{LapRRel} and Lemma \ref{CapRecursion} we have
		\begin{align*}
			c_{\lambda}(\mu)R^{2\mu_2}w&=D_\lambda (R^{2\mu_2}w)=\frac{R^2D_\nu\nabla^2}{2\lambda_2(2\lambda_1-2n-1)}\left( R^{2\mu_2}w \right)\\
			&=\frac{\mu_2(2\mu_1-2n-1)}{\lambda_2(2\lambda_1-2n-1)}R^2D_\nu R^{2\eta_2}w=\frac{\mu_2(\mu_1-(n+\frac{1}{2}))}{\lambda_2(\lambda_1-(n+\frac{1}{2}))}c_{\nu}(\eta)R^{2\mu_2}w.\qedhere
		\end{align*}
\end{proof}

	Indeed, applying Proposition \ref{CapEvalRecursive} recursively relates the eigenvalues of $D_\lambda$ to those of $D_\nu$, where $V_\nu$ is harmonic. That is, this reduces Problem \ref{CapEvalProb} to calculating the eigenvalues of $D_\nu$, for $\nu$ such that $V_\nu$ harmonic.

\begin{corollary}\label{RecurEvalEq}
	Let $\lambda,\mu\in \Lambda^*_n$. Let $\ell=\min(\lambda_2,\mu_2)$ and set $\nu=(\lambda_1-\ell,\lambda_2-\ell)$ and $\eta=(\mu_1-\ell,\mu_2-\ell)$. Then
		\begin{align}\label{RecEvalEquation}
			c_\lambda(\mu)&=\frac{\mu_2^{\underline{\lambda_2}}(\mu_1-(n+\frac{1}{2}))^{\underline{\lambda_2}}}{\lambda_2!(\lambda_1-(n+\frac{1}{2}))^{\underline{\lambda_2}}} c_\nu(\eta).
		\end{align}
\end{corollary}

\begin{proof}
	If $\mu_2<\lambda_2$, then the right hand side of Equation \eqref{RecEvalEquation} is $0$. Furthermore, Lemma \ref{CapRecursion} can be applied $\lambda_2$ times to write $d_{\lambda,\mu}D_\lambda=R^{2\lambda_2}D_\nu\nabla^{2\lambda_2},$ for some non-zero $d_{\lambda,\mu}\in \Z$. Thus for all $R^{2\mu_2}w\in V_\mu$, with $w$ harmonic, we have the following
		\begin{align*}
			c_\lambda(\mu)R^{2\mu_2}w&=D_\lambda R^{2\mu_2}w\\
			&=\frac{1}{d_{\lambda,\mu}}R^{2\lambda_2}D_\nu \nabla^{2\lambda_2}R^{2\mu_2}w=0,
		\end{align*}
	where the last equality follows by applying Lemma \ref{LapRRel} $\mu_2+1$ times. Thus for $\mu_2<\lambda_2$ we have Equation \eqref{RecEvalEquation}.
	
	If $\lambda_2\leq \mu_2$, then Lemma \ref{CapEvalRecursive} can be applied $\ell$-times to get
		\begin{align*}
			c_\lambda(\mu)&=\left( \prod_{i=1}^{\lambda_2}\frac{(\mu_2-i+1)(2(\mu_1-i)+1-2n)}{(\lambda_2-i+1)(2(\lambda_1-i)+1-2n)} \right)c_\nu(\eta).\qedhere
		\end{align*}
\end{proof}

Thus we focus on computing the eigenvalues of $D_\nu$, for $V_\nu$ harmonic. The set of irreducible harmonic submodules of $\mathcal{P}(V)$ are indexed by $$\Lambda_{n,\mathcal{H}}^*:=\lbrace (i,0)\mid 0\leq i\leq 2n+1 \rbrace.$$ The following lemma will allow us to relate the eigenvalues of $D_\nu$, for $\nu\in \Lambda_{n,\mathcal{H}}^*$, to Knop-Sahi polynomials $P_\nu^\rho$, see \cite{KnoSah96}.

\begin{lemma}\label{EigenProperties}
	Let $\nu \in \Lambda_{n,\mathcal{H}}^*$ and $\mu\in \Lambda_n^*$. Then $c_\nu(\mu)$ is a polynomial in $\mu_1,\mu_2$ such that
	\begin{enumerate}[(1)]
		\item $c_\nu(\mu)$ is symmetric in $x=\mu_1-\left(n+\frac{1}{2}\right)$ and $y=\mu_2$,
		\item $\deg(c_\nu(\mu))\leq |\nu|,$ and
		\item $c_\nu(\mu)=\delta_{\nu,\mu}$ for all $|\mu|\leq |\nu|$.
	\end{enumerate}
\end{lemma}

\begin{proof}
	Identify $C$ and $Z$ with their images under the map in Equation \eqref{AbsCapProb}. Theorem \ref{MainIntroTheorem} lets us write $D_\nu=\sum_{\lambda\in \Lambda_{|\nu|,n}} a_\lambda C^{\lambda_2}Z^{\lambda_1-\lambda_2}$ for some $a_\lambda\in \C$. Observe that the action of $C$ and $Z$ on $V_\mu$ are given by scalar multiplication by $(\mu_1-\mu_2)(2n+1+\mu_2-\mu_1)$ and $\mu_1+\mu_2$, respectively. Indeed, the action of $D_\nu$ on $V_\mu$ is given by
		\begin{align}\label{EvalAsPoly}
			c_\nu(\mu)&=\sum_{\lambda\in \Lambda_{|\nu|,n}} a_\lambda(\mu_1-\mu_2)^{\lambda_2}(2n+1+\mu_2-\mu_1)^{\lambda_2}(\mu_1+\mu_2)^{\lambda_1-\lambda_2}\nonumber\\
			&=\sum_{\lambda\in \Lambda_{|\nu|,n}}a_{\lambda}\left(x-y+n+\frac{1}{2}\right)^{\lambda_2}\left(y-x+n+\frac{1}{2}\right)^{\lambda_2}\left(x+y-n-\frac{1}{n}\right)^{\lambda_1-\lambda_2},
		\end{align}
	where $x=\mu_1-n-\frac{1}{2}$ and $y=\mu_2$. Condition (1) follows from Equation \eqref{EvalAsPoly}. Furthermore, every summand of \eqref{EvalAsPoly} has degree $|\lambda|$. This yields Condition (2). For Condition (3), the argument is similar to \cite[Lemma 5.4]{SalSah16} the point being $\bigcup_{k=0}^{2n+1} \mathbb{P}_{k,n}^*=\lbrace \mu\in \mathbb{P}\colon |\mu|\leq 2n+1\rbrace$. Indeed, if $|\mu|<|\nu|$ then the order of $D_\nu$ is $|\nu|$ and the degree of $V_\mu$ is $|\mu|$. That is, $c_\nu(\mu)=0$. If $|\nu|=|\mu|$, then $D_\nu$ results in a $\g$-module homomorphism $V_\mu\to V_\nu$ which can be non-zero only if $\mu=\nu$.
\end{proof}

The next proposition which follows from \cite[Prop. 3.4]{KnoSah96} gives an explicit expression for $P_\nu^\rho$ for $\nu \in \Lambda_{n,\mathcal{H}}^*$.

\begin{prop}
	Let $\nu\in \Lambda_{n,\mathcal{H}}^*$. Then
		\begin{align*}
			P_\nu^{\rho}&={n+\frac{1}{2}\choose \nu_1}^{-1}\sum_{i=0}^{\nu_1}{n+\frac{1}{2}\choose \nu_1-i}{n+\frac{1}{2}\choose i}\left(x+n+\frac{1}{2}-i\right)^{\underline{\nu_1-i}}\left(y\right)^{\underline{i}}.
		\end{align*}
\end{prop}

Lemma \ref{EigenProperties} and the uniqueness the Knop-Sahi polynomials yield the following.

\begin{prop}\label{CapHarmFormula}
	Let $\nu\in \Lambda_{n,\mathcal{H}}^*$ and $\mu\in \Lambda_n^*$. Then
		\begin{align*}
			c_\nu(\mu)&=\frac{1}{(n+\frac{1}{2})^{\underline{\nu_1}}}\sum_{i=0}^{\nu_1}{n+\frac{1}{2}\choose \nu_1-i}{n+\frac{1}{2}\choose i}\left(\mu_1-i\right)^{\underline{\nu_1-i}}\left(\mu_2\right)^{\underline{i}}.
		\end{align*}
\end{prop}

The proof of Theorem \ref{CapEvalThm} follows immediately.

\noindent \emph{Proof of Theorem \ref{CapEvalThm}.} Notice that $D_\nu$ corresponds to a harmonic irreducible module if and only if $\nu_2=0$. Thus we apply Proposition \ref{CapHarmFormula} and Corollary \ref{CapEvalRecursive}. When $\nu_2>0$ we use the fact that $P_\nu^\rho(x,y)=xyP_{\nu'}^\rho (x-1,y-1)$ for $\nu'=(\nu_1-1,\nu_2-1)$, see \cite[Prop. 2.3]{KnoSah96}. \qed

\bibliographystyle{alpha}
\bibliography{biblio}

\end{document}